\documentclass[11pt]{article}
\usepackage{amsmath,amsfonts,amsthm,amssymb, mathtools}
\usepackage{mathrsfs, graphicx,color,latexsym, tikz, calc,
}
\usepackage{tikz-cd}
\usepackage{graphicx}
\usepackage{epstopdf}
\usepackage{enumerate}
\usepackage{caption}
\usepackage{stmaryrd}

\usepackage{xcolor}
\usepackage{hyperref}
\hypersetup{
    colorlinks=true,
    citecolor=blue,
    linkcolor=blue,
    urlcolor=blue
}

\usetikzlibrary{shadows}
\usetikzlibrary{patterns,arrows,decorations.pathreplacing}
\numberwithin{equation}{section}
\newtheorem{theorem}{Theorem}[section]
\newtheorem{lemma}[theorem]{Lemma}
\newtheorem{definition}[theorem]{Definition}

\newtheorem{conjecture}[theorem]{Conjecture}
\newtheorem{corollary}[theorem]{Corollary}
\newtheorem{remark}[theorem]{Remark}

\newcommand{\F}{\mathcal{F}}

\allowdisplaybreaks[4]

\date{\today}
\title{\bf \Large An improved range for the maximum critically $t$-intersecting hypergraphs}
	\author{
{\small    Lu Lu$^1$,  \ \ Rongrong Lu$^1$, \ \  Qifan Wang$^{1,}$\footnote{Corresponding author.\newline{ \hspace*{5mm}Email addresses:}   \url{lulugdmath@163.com} (L. Lu), \url{lrr999a@163.com} (R. Lu), \url{jv01065499zai@163.com} (Q. Wang), \url{mathtzwu@163.com} (T. Wu). }, \ \ Tingzeng Wu$^2$
}\\[2mm]
\small $^1$School of Mathematics and Statistics, HNP-LAMA, Central South University\\
 \small Changsha, Hunan, 410083, China\\
\small  $^2$School of Mathematics and Statistics, Qinghai Minzu University\\
\small  Xining, Qinghai, 810007, China
}

\begin{document}
	\maketitle
	\begin{abstract}
		Let $k>t\ge 1$ be integers and set $d=k-t$.  A $k$-uniform hypergraph $\mathcal F$ is called $t$-intersecting if any two edges intersect in at least $t$ vertices, and is called $t$-critical if its minimum $t$-transversal has size $k$.  Frankl proved that, for $k\ge d^4$,
		$
		|\mathcal F|\le \binom{k+d}{d},
		$
		with equality only for the complete $k$-graph on $k+d$ vertices, and conjectured that the same conclusion should  hold when $k>c d^2$ for some constant $c$.  In this paper we confirm this conjecture for $c=30$. The proof  relies on Frankl's fixed-edge decomposition and F\"{u}redi's pseudo-sunflower method.

	\end{abstract}

	

	\section{Introduction}

Extremal Combinatorics is one of the main branches among modern combinatorics. It deals with the problems
about how big or how small a discrete structure can be given that it satisfies certain requirements. Many
natural mathematical objects and requirements can be investigated. As a natural extension of general graphs, the theory of
hypergraphs encounters more challenging problems and attracts much attention. In this paper, we are concerned with certain critically intersecting hypergraphs.

	Let $k> t\ge 1$ be integers. Let $\mathcal F$ be a $k$-\emph{uniform} hypergraph, that is $\F$ consists of distinct $k$-element sets. And let $\F$ be $t$-\emph{intersecting}, that is $F\cap F'\ge t$ for any $F,F'\in \F$.
	
	A set $T$ is called a $t$-\emph{transversal} of $\mathcal F$ if
	$
	|T\cap F|\ge t
	$ for all $F\in\mathcal F.
	$  Let $\tau_t(\mathcal F)$ be the minimum size of a $t$-transversal. Since $\F$ is $t$-intersecting, every edge of $\F$ is a $t$-transversal of $\F$. Then  $\tau_t(\mathcal F)\le k$. $\F$ is called $t$-\emph{critical} if $\tau_t(\mathcal F)= k$. When $t$ is clear, we use the term \emph{critical} as shorthand.
	
	Define
	\[
	m(k,t)=\max\{|\mathcal F|:\mathcal F\text{ is }t\text{-intersecting and }\text{critical}\}.
	\]
	Put
	$
	d=k-t.
	$
	A natural example is the complete $k$-graph on a set $Y$ of size $k+d$:
	\[
	\mathcal K(Y)=\binom{Y}{k}.
	\]
	It is easy to check that any two $k$-subsets of $Y$ intersect in at least
	$
	2k-(k+d)=k-d=t
	$
	vertices, and no set of size $k-1$ is a $t$-transversal. Hence $\mathcal K(Y)$ is $t\text{-intersecting and }\text{critical}$, then
	\[
	m(k,k-d)\ge |\mathcal K(Y)|=\binom{k+d}{d}.
	\]
	
	The study of critical intersecting hypergraphs dates back to Erd\H{o}s and Lov\'asz \cite{EL}, cf. also \cite{Lovasz1975}.  They considered the case $t=1$ which is closely related to the maximum size of an intersecting $k$-graph with covering number $k$, and they obtained   $k!(e-1)\le m(k,1)\le k^k$.
 For other related results, one may refer to \cite{Tuza,AT,Frankl2019,Zakharov2024, FOT1996}.
	
	
	 Frankl recently proved the following theorem \cite{FranklCritical}.
	\begin{theorem}[Frankl \cite{FranklCritical}] \label{FranklCritical}
		For positive integers $d$ and $k\ge d^4$,
		$$
		m(k,k-d)=\binom{k+d}{d}.
		$$
		Moreover, the equality holds for the complete $k$-graph on $k+d$ vertices.
		\end{theorem}
	
	 Frankl also  proved the case $d=1$, $k\ge 1$ and the case $d=2$, $k\ge3$.
	 \begin{theorem}[Frankl \cite{FranklCritical}]
	 	$m(k, k-2) = \binom{k+2}{2} \text{ for } k \ge 3.$
	 	Moreover,  for $k \ge 4$,  equality holds only if $\mathcal{F} = \binom{Y}{k}$ for some $(k+2)$-set $Y$.
	 \end{theorem}
	
	 Furthermore, Frankl \cite{FranklCritical} stated that the threshold in Theorem  \ref{FranklCritical} can be improved to $k>2d^{3.5}$ with some efforts,  and he proposed the following conjecture.
	 \begin{conjecture}[Frankl \cite{FranklCritical}] \label{ConjF}
	 	There exists a constant $c$ such that for $k>c d^2$,
	 	\[
	 	m(k,k-d)=\binom{k+d}{d}.
	 	\]
	 \end{conjecture}
In this paper, we confirm
	  Conjecture \ref{ConjF}   for $c=30$,  and we also characterize  the extremal graphs.
	\begin{theorem}\label{thm:main}
		For positive integers $d$ and $k\ge30d^2$,
		\[
		m(k,k-d)=  \binom{k+d}{d}.
		\]
		Moreover, equality holds if and only if $\mathcal F$ is the complete $k$-graph on $k+d$ vertices.
	\end{theorem}
\begin{remark}
	While we have confirmed Frankl's Conjecture 1.3 for $c=30$, the exact threshold is still unknown. We believe it will be of great interest to find it.
\end{remark}

This paper is organized as follows. In Section  \ref{sec2}, we introduce our tools, including F\"{u}redi's pseudo-sunflower method and  Frankl's fixed-edge decomposition. In  Section  \ref{sec3}, we prove our main result.
	
	\section{Preliminaries}\label{sec2}

We describe two ingredients that we need  for our proof in this section, and present several basic results.
	\subsection{Pseudo-sunflowers}
	Our proof is based on F\"{u}redi's pseudo-sunflower  method.
	\begin{definition}
		Let $P_1,\ldots,P_s$ be distinct sets. $\{P_1,\ldots,P_s\}$ is called a pseudo-sunflower with center $C$ and size $s$, if $C$ is a proper subset of $P_1$ and the sets
		$
		P_1\setminus C,\ldots,P_s\setminus C
		$
		are pairwise disjoint.
		
	\end{definition}
	F\"{u}redi \cite{Furedi} (see also \cite{FranklPseudo}) proved the following important result.
	\begin{theorem}[F\"{u}redi \cite{Furedi}, Frankl \cite{FranklPseudo}]\label{thm:furedi}
		Suppose that $s\ge 1$ and  $\mathcal H$ is an $\ell$-uniform hypergraph.  If $\mathcal H$ contains no pseudo-sunflower of size $s+1$, then
		\[
		|\mathcal H|\le s^\ell.
		\]
		
		\end{theorem}
	
	 \subsection{Frankl's decomposition with fixed-edge}
	Let $\mathcal F$ be a $t$-intersecting, critical $k$-uniform hypergraph and put $d=k-t$.  Fix an edge
	$
	E\in\mathcal F.
	$
	For each $P\subseteq E$, define
	\[
	\mathcal F(P)=\{F\setminus E:F\in\mathcal F,\ E\setminus F=P\}.
	\]
	Note that $|F\setminus E|=|E\setminus F|=k-|F\cap E|$. Then if $|P|=\ell$, every member of $\F(P)$ has size $\ell$. Besides, $|F\cap E|\ge t$ implies that $|E\setminus F|\le k-t=d$. It means that $\mathcal F(P)=\emptyset$ for $|P|>d$. Hence
	\begin{equation}\label{eq:fixed-edge-sum}
		|\mathcal F|
		=
		\sum_{\ell=0}^{d}\sum_{P\in\binom{E}{\ell}} |\mathcal F(P)|.
	\end{equation}
	
	Define the top layer support
	$$
	\mathcal G=
	\left\{G\in\binom{E}{d}:\mathcal F(G)\ne\emptyset\right\}.
	$$
	A set $P\subseteq E$ is called \textit{rich} if
	$
	\left|\bigcup\mathcal F(P)\right|>d.
	$
	Let
$$
	\mathcal R=\{P\subseteq E:P\text{ is rich}\}.
	$$
	If $P$ is not rich and $|P|=\ell$, then all members of $\mathcal F(P)$ lie in a  surrounding set of size at most $d$, so
	\begin{equation}\label{eq:nonrich-bound}
		|\mathcal F(P)|\le \binom{d}{\ell}.
	\end{equation}
	Frankl \cite{FranklCritical} showed the following facts.
	\begin{lemma}[Frankl \cite{FranklCritical}]\label{lem:frankl-facts}
		With the notation above, the following statements hold.
		\begin{enumerate}[(1).]
			\item For every $P\subseteq E$,   $\mathcal F(P)$ contains no pseudo-sunflower of size $d+1$.
			\item For disjoint $D_1,D_2\in\binom{E}{d}$,
			$
			|\mathcal F(D_1)|\,|\mathcal F(D_2)|\le 1.
			$
			\item The families $\mathcal G$ and $\mathcal R$ are cross-intersecting.
			\item If $\mathcal R=\emptyset$, then
			\[
			|\mathcal F|\le \binom{k+d}{d}.
			\]
			Moreover, if $k>2d$, then equality holds only for the complete $k$-graph on $k+d$ vertices.
		\end{enumerate}
	\end{lemma}

	 \section{Proof of the main theorem}\label{sec3}

In this section, we will give  the proof of  Theorem \ref{thm:main} step by step.
\smallskip

	\textbf{Firstly, we give a kind of compression.}

\smallskip

	For two $\ell$-sets $X,Y$, we define the Johnson   distance between them as
	\[
	\delta(X,Y)=\ell-|X\cap Y|=|X\setminus Y|.
	\]
	Thus $|X\cap Y|\ge \ell-r$ is equivalent to $\delta(X,Y)\le r$.
	\begin{lemma}\label{lem:two-balls}
		Let integers $d\ge 2$, $1\le \rho\le \ell\le d$ and let $\ell$-sets $X, Y$ satisfy
		\[
		\delta(X,Y)=s,
		\quad
		\rho\le s\le 2\rho.
		\]
		Let $\mathcal H$ be an  $\ell$-uniform hypergraph containing no pseudo-sunflower of size $d+1$.  Suppose further that every $H\in\mathcal H$ satisfies
		\[
		\delta(H,X)\le \rho,
		\quad
		\delta(H,Y)\le \rho.
		\]
		Then
		\[
		|\mathcal H|\le (8d)^\rho.
		\]
	\end{lemma}
	\begin{proof}
		Put
		\[
		C=X\cap Y,
		\qquad
		X_0=X\setminus Y,
		\qquad
		Y_0=Y\setminus X.
		\]
		Then $|X_0|=|Y_0|=s$ and $|C|=\ell-s$.
		
		Let $U=C\setminus H$. For $H\in\mathcal H$, we can decompose $H$ into four parts:
		\[
		W=C\setminus U=C\cap H,
		\qquad
		I=H\cap X_0,
		\qquad
		J=H\cap Y_0,
		\qquad
		Z=H\setminus (X\cup Y).
		\]
		Denote
		\[
		\alpha:=|U|,
		\quad
		\beta:=|I|,
		\quad
		\gamma:=|J|,
		\quad
		\zeta:=|Z|.
		\]
		Note that
		$
		H=W\cup I\cup J\cup Z
		$
		and $|H|=\ell$, we have $\ell-s-\alpha+\beta+\gamma+\zeta=\ell$, that is
		\begin{equation}\label{eq:parts-size}
			\beta+\gamma+\zeta=s+\alpha.
		\end{equation}
		Moreover,
		$
		|H\cap X|=|W|+|I|=\ell-s-\alpha+\beta.
		$
		Therefore
		$
		\delta(H,X)=s+\alpha-\beta\le \rho,
		$
		so
		\begin{equation}\label{eq:beta-bound}
			\beta\ge s+\alpha-\rho.
		\end{equation}
		Similarly,
		\begin{equation}\label{eq:gamma-bound}
			\gamma\ge s+\alpha-\rho.
		\end{equation}
		Combining \eqref{eq:parts-size}, \eqref{eq:beta-bound}, and \eqref{eq:gamma-bound}, we obtain
		\[
		s+\alpha-\zeta
		=\beta+\gamma
		\ge 2(s+\alpha-\rho).
		\]
		Thus
		\begin{equation}\label{eq:free-bound}
			\alpha+\zeta\le 2\rho-s.
		\end{equation}
		
		Now we count.  For fixed $U,I,J$, $H$ is determined only by $Z$.  The possible $Z$ form a $\zeta$-uniform family containing no pseudo-sunflower of size $d+1$. Otherwise suppose $\{Z_0,Z_1,\dots,Z_d\}$ is a pseudo-sunflower with center $M$, then the corresponding $\{H_0,H_1,\dots,H_d\}$ where $H_i=Z_i\cup W\cup I\cup J$ would form a pseudo-sunflower with center $M\cup W\cup I\cup J$.  By Theorem \ref{thm:furedi}, the number of possible $Z$'s is at most $d^\zeta\le d^{2\rho-s-\alpha}$.
		
		Note that the pair $(I,J)$ has at most $2^s2^s=4^s$ choices.  By \eqref{eq:free-bound}, $\alpha\le 2\rho-s$, and for fixed $\alpha$ the set $U\subseteq C$ has at most
		$
		\binom{\ell-s}{\alpha}\le\binom{\ell}{\alpha}\le \ell^\alpha\le d^\alpha
		$
		choices.  Therefore
		\[
		|\mathcal H|
		\le
		4^s\sum_{\alpha=0}^{2\rho-s} d^\alpha d^{2\rho-s-\alpha}
		=
		4^s(2\rho-s+1)d^{2\rho-s}.
		\]
		Denote
		$
		m:=2\rho-s.
		$
		Then $0\le m\le \rho$ and the last expression is
		$
		4^{2\rho-m}(m+1)d^m.
		$
		We compare it to $(8d)^\rho$.  Their ratio is
		\[
		\frac{4^{2\rho-m}(m+1)d^m}{(8d)^\rho}
		=
		(m+1)2^{\rho-2m}d^{m-\rho}.
		\]
		Since $d\ge 2$ and $m\le \rho$,
		$
		d^{m-\rho}\le 2^{m-\rho}.
		$
		Hence the ratio is at most
		$
		(m+1)2^{-m}\le 1.
		$
		Therefore
		\[
		|\mathcal H|\le (8d)^\rho.
		\]
	\end{proof}
	The \emph{diameter} of $\mathcal{H}$ is defined as $ diam(\mathcal{H})=\max\left\{\delta(H,H'):H,H'\in \mathcal{H}\right\}$.
	\begin{corollary}\label{cor:small-diameter}
		Suppose that $d\ge 2$, $\ell\le d$. Let $\mathcal H$ be an $\ell$-uniform hypergraph containing no pseudo-sunflower of size $d+1$.  If
		$
		diam(\mathcal{H})\le r,
	$
		then
		$
		|\mathcal H|\le (8d)^r.
		$
	\end{corollary}
	\begin{proof}
		The case $|\mathcal H|=1$ is trivial.  Suppose that $|\mathcal H|\ge2$, we can choose $X,Y\in\mathcal H$ with
		\[
		s=\delta(X,Y)=diam(\mathcal H).
		\]
		Then $1\le s\le r$.  Every $H\in\mathcal H$ satisfies
		$
		\delta(H,X)\le s,
		$ $
		\delta(H,Y)\le s.
		$
	We can get
		\[
		|\mathcal H|\le (8d)^s\le (8d)^r
		\] by applying Lemma \ref{lem:two-balls} with $\rho=s$.
	\end{proof}
	\begin{theorem}\label{thm:compression}
		Suppose that $d\ge 2$, $\ell\le d$. Let $\mathcal{A}$, $\mathcal{B}$ be $\ell$-uniform hypergraphs.
		Assume that both $\mathcal A$ and $\mathcal B$ contain no pseudo-sunflower of size $d+1$.  If $\mathcal A$ and $\mathcal B$ are cross-$(\ell-r)$-intersecting,
		then
		\[
		\min\{|\mathcal A|,|\mathcal B|\}\le (8d)^r.
		\]
	\end{theorem}
	
	\begin{proof}
		For $r=0$, $\mathcal A$ and $\mathcal B$ are cross-$\ell$-intersecting, implying that every $A\in\mathcal A$ and $B\in\mathcal B$ must be equal. Then $|\mathcal{A}|\le 1$ and $|\mathcal{B}|\le 1$.  So we may assume $r\ge 1$. Furthermore, the case when $\mathcal{B}$ is empty is trivial. Then we can further assume $\mathcal{B}$ is nonempty.
		
		If ${\rm{diam}} (\mathcal A)\le r$, then Corollary \ref{cor:small-diameter} gives that
		$
		|\mathcal A|\le (8d)^r.
		$
		Otherwise we can choose $A_1,A_2\in\mathcal A$ with
		$
		s=\delta(A_1,A_2)>r.
	$
		For every $B\in\mathcal B$, the assumption $\mathcal{A}$ and $\mathcal{B}$  are cross-$(\ell-r)$-intersecting gives that
	$$
		\delta(B,A_1)=\ell-|B\cap A_1|\le \ell-(\ell-r)=r,
		$$
		similarly $\delta(B,A_2)\le r$.
		Then the triangle inequality for the Johnson distance implies $s\le 2r$.  Thus
		\[
		r<s\le 2r.
		\]
		Applying Lemma \ref{lem:two-balls} to $\mathcal B$ with $X=A_1$, $Y=A_2$, and $\rho=r$, we obtain
		\[
		|\mathcal B|\le (8d)^r.
		\]
		This proves the theorem.
	\end{proof}
	
\smallskip

\textbf{Secondly, we consider the case when $|\mathcal F(P)|$ is large.}

\medskip

	Put $M(P)=|\mathcal F(P)|.
	$ By Lemma \ref{lem:frankl-facts} and Theorem \ref{thm:furedi}, for $|P|=\ell$ we have that $\mathcal F(P)$ contains no pseudo-sunflower of size $\ell+1$ and
	\begin{equation}\label{eq:MP-rough}
		M(P)\le d^\ell.
	\end{equation}
	Denote
	$
	L_r:=(8d)^r
	$ for $r\ge0$, $a=d-\ell$.
	For a fixed layer $P$ with
	$
	|P|=\ell=d-a
	$
, $s\ge 0$, define
	\[
	\mathcal H_{\ell,s}
	=
	\left\{P\in\binom{E}{\ell}:M(P)>L_{a+s}\right\}.
	\]
	\begin{lemma}\label{lem:large-fibres-intersect}
		For $0\le s\le \ell-1$, the family $\mathcal H_{\ell,s}$ is $(s+1)$-intersecting.
	\end{lemma}
	\begin{proof}
		Firstly, suppose that $a=s=0$.  Then $\ell=d$.  For disjoint $P,P'\in\mathcal H_{d,0}$, note that
		\[
		M(P)>L_0=1,
		\qquad
		M(P')>L_0=1,
		\]
		contradicting Lemma \ref{lem:frankl-facts}(2).  Hence $\mathcal H_{d,0}$ is intersecting.
		
		Now assume $a+s\ge 1$.  Suppose that there are $P,P'\in\mathcal H_{\ell,s}$ satisfying
		$
		|P\cap P'|\le s.
		$
		Note that $\mathcal F(P)$ and $\mathcal F(P')$ are both nonempty. Then take arbitrary
		$
		Q\in\mathcal F(P)
		$, $
		Q'\in\mathcal F(P').
		$
		The corresponding edges of $\mathcal F$ are
		$
		F=Q\cup(E\setminus P),
		$ $
		F'=Q'\cup(E\setminus P').
		$
		Since $\mathcal F$ is $t$-intersecting and $t=k-d$,
		$
		|F\cap F'|\ge k-d.
		$
		Note that $Q$, $Q'$ are both disjoint with $E$, then
		\[
		|F\cap F'|
		=
		|Q\cap Q'|+|E\setminus(P\cup P')|.
		\]
		Since $|E\setminus(P\cup P')|=k-|P\cup P'|$,
		\[
		|Q\cap Q'|
		\ge
		k-d-(k-|P\cup P'|)=|P\cup P'|-d.
		\]
		Using $|P\cup P'|=2\ell-|P\cap P'|\ge 2\ell-s$ and $\ell=d-a$, we get
		\[
		|Q\cap Q'|
		\ge
		2\ell-s-d
		=\ell+(d-a)-s-d=
		\ell-(a+s).
		\]
		Thus the two fibers $\mathcal F(P)$ and $\mathcal F(P')$ are cross-$(\ell-(a+s))$-intersecting.
		
		If $a+s>\ell$, then by \eqref{eq:MP-rough},
		$
		M(P)\le d^\ell<L_{a+s},
		$
		contradicting $P\in\mathcal H_{\ell,s}$.  If $a+s\le \ell$.  Theorem \ref{thm:compression} gives
		$
		\min\{M(P),M(P')\}\le L_{a+s},
		$
		contradicting $P,P'\in\mathcal H_{\ell,s}$.  This proves the lemma.
	\end{proof}
	Recall Wilson's exact version of the Erd\H{o}s--Ko--Rado theorem \cite{EKR,Wilson}.
	
	\begin{theorem}[Wilson \cite{Wilson}]\label{thm:wilson}
		Let $\mathcal A\subseteq\binom{[n]}{\ell}$ be $q$-intersecting.  If
		\[
		n\ge (q+1)(\ell-q+1),
		\]
		then
		\[
		|\mathcal A|\le \binom{n-q}{\ell-q}.
		\]
	\end{theorem}
	For $q=s+1$ and $n=k$,  Theorem \ref{thm:wilson} implies the following bound.
	\begin{corollary}\label{cor:H-bound}
		Suppose that $k\ge 30d^2$ and $d\ge 3$.  Then, for $0\le s\le \ell-1\le d-1$,
		\[
		|\mathcal H_{\ell,s}|
		\le
		\binom{k-s-1}{\ell-s-1}.
		\]
	\end{corollary}
	\begin{proof}
		By Lemma \ref{lem:large-fibres-intersect}, $\mathcal H_{\ell,s}$ is $(s+1)$-intersecting.
		Since $\ell\le d$ and $0\le s\le \ell-1$, we have
		\[
		(s+2)(\ell-s)
		\le
		(d+1)d
		<30d^2
		\le k.
		\]
		The result immediately follows from Theorem \ref{thm:wilson}.
	\end{proof}
\medskip

\textbf{Thirdly, we need some estimations.}

\medskip

	Define
	\begin{equation}\label{eq:X-def}
		X=
		\sum_{\ell=0}^{d}
		\sum_{P\in\binom{E}{\ell}}
		\left(M(P)-\binom{d}{\ell}\right)_+.
	\end{equation}
	Then, for rich sets, we have
	\begin{equation}\label{eq:rich-le-X}
		\sum_{P\in\mathcal R}
		\left(M(P)-\binom{d}{|P|}\right)_+
		\le X.
	\end{equation}
	\begin{lemma}\label{lem:layer-cake}
		Suppose that $d\ge 3$, $k\ge 30d^2$, and denote
	$
		\lambda:=\frac{8d^2}{k-d}.
	$
		Then we have
		\[
		X
		\le
		\left(
		\frac{2\lambda}{1-\lambda}
		+
		\frac{\lambda^2}{(1-\lambda)^2}
		\right)
		\binom{k}{d}.
		\]
		Consequently,
		\[
		X\le \frac{400}{441}\binom{k}{d}.
		\]
	\end{lemma}
	\begin{proof}
		We first consider the top layer $\ell=d$. 

 Here $a=0$ and $\binom{d}{d}=1$.  Since $M(P)\le d^d\le L_d$, $L_0=1$, a simple stratify gives that
		\[
		(M(P)-1)_+
		\le
		\sum_{s=0}^{d-1}(L_{s+1}-L_s)\mathbf 1_{\{M(P)>L_s\}},
		\]
		where we denote by $\mathbf{1}\{E\}$ the indicator function of an event $E$,
		that is, $\mathbf{1}_{\{E\}} = 1$ if $E$ holds and $0$ otherwise.
		
		Summing over $P\in\binom{E}{d}$ and using Corollary \ref{cor:H-bound}, we get
		\[
		\begin{aligned}
			\sum_{P\in\binom{E}{d}}(M(P)-1)_+
			&\le
			\sum_{s=0}^{d-1}(L_{s+1}-L_s)|\mathcal H_{d,s}|  \\
			&\le
			\sum_{s=0}^{d-1}L_{s+1}\binom{k-s-1}{d-s-1}.
		\end{aligned}
		\]
		Put $r=s+1$.  Since
		$
		\binom{k-r}{d-r}\le \binom{k}{d-r}
		$
		and
		$
		\frac{\binom{k}{d-r}}{\binom{k}{d}}=\prod\limits_{i=0}^{r-1}\frac{d-i}{k-d+i+1}
		\le
		\left(\frac{d}{k-d}\right)^r,
		$
		we have
		\begin{equation}\label{eq:binom-estimate}
		L_r\binom{k-r}{d-r}
		\le
		(8d)^r\left(\frac{d}{k-d}\right)^r\binom{k}{d}
		=
		\lambda^r\binom{k}{d}.
		\end{equation}
		Therefore the contribution of the top layer is at most
		\begin{equation}\label{eq:top-layer}
			\frac{\lambda}{1-\lambda}\binom{k}{d}.
		\end{equation}
		
		Now consider the lower layer $\ell=d-a<d$, where $a\ge 1$.

 Since $M(P)\le d^\ell\le L_\ell$, then
		for each $P\in\binom{E}{\ell}$,
		\[
		M(P)
		\le
		L_a+
		\sum_{\substack{s\ge 0\\ a+s<\ell}}
		(L_{a+s+1}-L_{a+s})\mathbf 1_{\{M(P)>L_{a+s}\}}.
		\]
		Using the simple bound $\left(M(P)-\binom{d}{\ell}\right)_+\le M(P)$, we get the contribution of this layer is at most
		\[
		L_a\binom{k}{d-a}
		+
		\sum_{\substack{s\ge 0\\ a+s<\ell}}
		L_{a+s+1}|\mathcal H_{\ell,s}|.
		\]
	For the second term put
		$
		r=a+s+1.
		$
		By Corollary \ref{cor:H-bound},
		\[
		|\mathcal H_{\ell,s}|
		\le
		\binom{k-s-1}{\ell-s-1}\le \binom{k}{d-r}.
		\]
		Similar to \eqref{eq:binom-estimate},
		$
		L_a\binom{k}{d-a}
		\le
		\lambda^a\binom{k}{d}
		$
		 and $
		L_{a+s+1}|\mathcal H_{\ell,s}|
		\le
		L_r\binom{k}{d-r}
		\le
		\lambda^r\binom{k}{d}.
		$
		Summing over all $a\ge 1$ and over all $r\ge a+1$, we obtain the total lower-layer contribution
		\begin{equation}\label{layer-bound}
		\left(
		\sum_{a\ge 1}\lambda^a
		+
		\sum_{a\ge 1}\sum_{r\ge a+1}\lambda^r
		\right)
		\binom{k}{d}.
		\end{equation}
		Note that $\lambda\le\frac{8d^2}{29d^2}\le\frac{8}{29}<1$, then $\sum\limits_{a\ge 1}\lambda^a=\frac{\lambda}{1-\lambda}$. Besides, $$\sum_{a\ge 1}\sum_{r\ge a+1}\lambda^r=\sum_{a\ge 1}{\frac{\lambda^{a+1}}{1-\lambda}}=\frac{\lambda}{1-\lambda}\sum\limits_{a\ge 1}\lambda^a=\frac{\lambda^2}{(1-\lambda)^2}.$$
		Therefore \eqref{layer-bound} is exactly
		\[
		\left(
		\frac{\lambda}{1-\lambda}
		+
		\frac{\lambda^2}{(1-\lambda)^2}
		\right)
		\binom{k}{d}.
		\]
		Then the total bound is given by combining \eqref{eq:top-layer}.
		
		Finally, since $\lambda\le\frac{8}{29}$,
		then
		$
		\frac{\lambda}{1-\lambda}\le \frac{8}{21}
		$
		and hence
		$
		\frac{2\lambda}{1-\lambda}
		+
		\frac{\lambda^2}{(1-\lambda)^2}
		\le
		\frac{16}{21}+\frac{64}{441}
		=
		\frac{400}{441}.
		$
	\end{proof}
\begin{corollary}\label{cor:charging}
	Suppose that $d\ge 3$ and $k\ge 30d^2$.  Let
	$
	\mathcal{T}=\binom{E}{d}\setminus\mathcal G.
	$
	Then
	\[
	\sum_{P\in\mathcal R}
	\left(M(P)-\binom{d}{|P|}\right)_+
	\le
	|\mathcal{T}|.
	\]
	Moreover, if $\mathcal R\ne\emptyset$, the inequality is strict.
\end{corollary}
\begin{proof}
	The case when $\mathcal R=\emptyset$ is trivial.  Now we assume $\mathcal R\ne\emptyset$, and choose $R_0\in\mathcal R$.  By Lemma \ref{lem:frankl-facts}(3), every $G\in\mathcal G$ intersects $R_0$.  Hence every $d$-set contained in $E\setminus R_0$ is not in $\mathcal G$, it means that
	\[
	\binom{E\setminus R_0}{d}\subseteq \mathcal{T}.
	\]
	Since $|R_0|\le d$, we get
	\begin{equation}\label{eq:T-lower}
		|\mathcal{T}|\ge \binom{k-d}{d}.
	\end{equation}
	Furthermore,
	\[
	\frac{\binom{k-d}{d}}{\binom{k}{d}}
	=
	\prod_{j=0}^{d-1}\frac{k-d-j}{k-j}.
	\]
Since $k\ge30d^2$ and $\frac{k-d-j}{k-j}\ge\frac{k-2d}{k}=1-\frac{2d}{k}\ge0$, by applying Weierstrass’s product inequality,  we get
	\[
	\frac{\binom{k-d}{d}}{\binom{k}{d}}
	\ge
	\left(1-\frac{2d}{k}\right)^d
	\ge
	1-\frac{2d^2}{k}
	\ge
	\frac{14}{15}.
	\]
	Thus
	\begin{equation}\label{eq:T-ratio}
		|\mathcal{T}|\ge \frac{14}{15}\binom{k}{d}.
	\end{equation}
	On the other hand, by Lemma \ref{lem:layer-cake} and \eqref{eq:rich-le-X},
	\[
	\sum_{P\in\mathcal R}
	\left(M(P)-\binom{d}{|P|}\right)_+
	\le
	X
	\le
	\frac{400}{441}\binom{k}{d}< \frac{14}{15}\binom{k}{d}.
	\]
	This proves the result.
\end{proof}

\smallskip

Now we are ready to present the proof of Theorem \ref{thm:main}.

\begin{proof}[Proof of Theorem \ref{thm:main}]
	The cases $d=1$ and $d=2$ exactly follow from Frankl's results \cite{FranklCritical}.  Therefore we assume that $d\ge 3$.
	
	Fix an edge $E\in\mathcal F$, and we adopt the notation introduced above.  By \eqref{eq:fixed-edge-sum}, we have
	\[
	|\mathcal F|
	=
	\sum_{\ell=0}^{d}\sum_{P\in\binom{E}{\ell}} M(P).
	\]
	For non-rich $P$, we use the bound $\binom{d}{|P|}$ by \eqref{eq:nonrich-bound}.

For $P$ at the top layer, only the sets in $\mathcal G$ contribute $1=\binom{d}{d}$.  Therefore
	\begin{align*}
		|\mathcal F|
		&\le
		\sum_{\ell=0}^{d-1}\binom{k}{\ell}\binom{d}{\ell}
		+|\mathcal G|
		+
		\sum_{P\in\mathcal R}
		\left(M(P)-\binom{d}{|P|}\right)_+ .
	\end{align*}
	Note that
	$
	|\mathcal G|=\binom{k}{d}-|\mathcal{T}|.
	$
	Consequently,
	\begin{align*}
		|\mathcal F|
		&\le
		\sum_{\ell=0}^{d}\binom{k}{\ell}\binom{d}{\ell}
		-|\mathcal{T}|
		+
		\sum_{P\in\mathcal R}
		\left(M(P)-\binom{d}{|P|}\right)_+ .
	\end{align*}
	By Corollary \ref{cor:charging}, the last two terms are at most zero.  Thus
	\[
	|\mathcal F|
	\le
	\sum_{\ell=0}^{d}\binom{k}{\ell}\binom{d}{\ell}.
	\]
	The Vandermonde's identity gives
	\[
	\sum_{\ell=0}^{d}\binom{k}{\ell}\binom{d}{\ell}
	=
	\sum_{\ell=0}^{d}\binom{k}{\ell}\binom{d}{d-\ell}
	=
	\binom{k+d}{d}.
	\]
	Thus
	$
	|\mathcal F|\le \binom{k+d}{d}.
	$
	
	It remains to discuss the equality.  If $\mathcal R\ne\emptyset$, Corollary \ref{cor:charging} gives the strict inequality.  Hence equality can occur only when $\mathcal R=\emptyset$.  Lemma \ref{lem:frankl-facts}(4) implies that, since $k\ge 30d^2>2d$, equality holds only when
	\[
	\mathcal F=\binom{Y}{k}
	\]
	for some set $Y$ of size $k+d$,
	completing the proof.
\end{proof}

	\section*{Declaration of competing interest}
	The authors declare that they have no conflicts of interests to this paper.
	
	\section*{Data availability}
	No data was used for the research described in the paper.
	
	 \section*{Acknowledgement}
	 Lu Lu is supported by National Natural Science Foundation of China (No. 12371362). T. Wu was supported by the NSFC (No. 12261071) and NSF of Qinghai Province (No. 2025-ZJ-902T).
	 
	 \section*{Declaration of AI usage}
The authors acknowledge the use of AI tools during the exploratory stage of this project. All
mathematical arguments and proofs in the final manuscript were checked and written by the
authors.



\begin{thebibliography}{99}

			\bibitem{EKR}
			P. Erd\H{o}s, C. Ko, and R. Rado,
			Intersection theorems for systems of finite sets,
			\emph{Quart. J. Math. Oxford Ser. } 12(2) (1961), 313--320.
			
			\bibitem{EL}
			P. Erd\H{o}s and L. Lov\'asz,
			Problems and results on 3-chromatic hypergraphs,
			in \emph{Infinite and Finite Sets}, North-Holland, Amsterdam, 1974, 609--627.
		
			
			
			
			\bibitem{FranklPseudo}
			P. Frankl,
			 Pseudo sunflowers,
			 \emph{European J. Combin.} 104 (2022), Paper No. 103553.
			
			\bibitem{FranklCritical}
			P. Frankl,
			 Critically intersecting hypergraphs,
			 \emph{European J. Combin.} 132 (2026), Paper No. 104286.


\bibitem{Frankl2019}
	P.~Frankl, A near exponential improvement on a bound of Erd\H{o}s and Lov\'{a}sz, {\em Combin.\ Probab.\ Comput.} 28 (2019), 1--7.

	
\bibitem{FOT1996}
P.~Frankl, K.~Ota, and N.~Tokushige, Covers in uniform intersecting families and a counterexample to a conjecture of Lov\'{a}sz, {\em J.\ Combin.\ Theory Ser.\ A} 74 (1996), 33--42.

	
\bibitem{Furedi}
Z. F\"uredi,
	On maximal intersecting families of finite sets,
\emph{J. Combin. Theory Ser. A} 28 (1980), 282--289.

\bibitem{Lovasz1975} L. Lov\'asz, On the minimax theorems of combinatorics, {\em Mat. Lapok} 26 (1975), 209--264.
			
			

	\bibitem{Tuza}
Z. Tuza, Inequalities for minimal covering sets in set systems of given rank, {\em Discrete Appl.\ Math.} 51 (1994), 187--195.
					\bibitem{AT} A.~Arman and T.~Retter, An upper bound for the size of a $k$-uniform intersecting family with covering number $k$, {\em J.\ Combin.\ Theory Ser.\ A} 147 (2017), 18--26.



\bibitem{Zakharov2024}
			D.~Zakharov, On the size of maximal intersecting families, {\em Combin.\ Probab.\ Comput.} 33 (2024), 32--49.

\bibitem{Wilson}
			R. M. Wilson,
			The exact bound in the Erd\H{o}s--Ko--Rado theorem,
			\emph{Combinatorica} 4 (1984), 247--257.

	\end{thebibliography}
\end{document}